\documentclass[reqno, 12pt]{amsart}
\usepackage{amsmath}
 \usepackage{amssymb}
\usepackage{amsthm}
\usepackage{times}
\usepackage{latexsym}
\usepackage[mathscr]{eucal}

\numberwithin{equation}{section}
 
  \newtheorem{theorem}{Theorem}[section]
  \newtheorem{proposition}[theorem]{Proposition}
  \newtheorem{lemma}[theorem]{Lemma}
  \newtheorem{corollary}[theorem]{Corollary}

\title[On doubly twisted product immersions]{On doubly twisted product immersions}

\author[Abdoul Salam Diallo, Fortun\'{e} Massamba ]{Abdoul Salam Diallo*, Fortun\'{e} Massamba**}

\newcommand{\acr}{\newline\indent}

\address{\llap{*\,} School of Mathematics, Statistics and Computer Science\acr
 University of KwaZulu-Natal\acr
 Private Bag X01, Scottsville 3209\acr
South Africa  \acr
and \acr
Universit\'e Alioune Diop de Bambey\acr
UFR SATIC, D\'epartement de Math\'ematiques\acr
B. P. 30, Bambey, S\'en\'egal}
\email{Diallo@ukzn.ac.za, abdoulsalam.diallo@uadb.edu.sn}
\address{\llap{**\,} School of Mathematics, Statistics and Computer Science\acr
 University of KwaZulu-Natal\acr
 Private Bag X01, Scottsville 3209\acr
South Africa}
\email{massfort@yahoo.fr, Massamba@ukzn.ac.za}

\thanks{} 



\begin{document}
 
\begin{abstract}  
Some basic geometric properties of doubly twisted product immersions are established.
\end{abstract}

\maketitle

\section{Introduction}

One is often interested in the decomposition of mathematical objects as a first step for a classification. A well-known theorem of de Rham gives sufficient conditions for a Riemannian manifold to be a Riemannian product. Also well-known is Moore's theorem which gives a sufficient condition for an
isometric immersion into a Euclidean space to split into a product immersion. An analogue for de Rham's theorem for warped product metrics was
given by Hiepko \cite{hi}. Also, an analogue of Moore's theorem for warped product immersions was given by Nolker \cite{nolker}.

Intrinsic and extrinsic invariants are very powerful to study submanifolds of Riemannian manifolds, and to establish relationship between them is one of
the must fundamental problems in submanifolds theory \cite{chenbook}. In this context, Chen \cite{chen2,chen5} proved some basic inequalities
for warped products isometrically immersed in arbitrary Riemannian manifolds. Corresponding inequalities have been obtained for doubly warped
product submanifolds into arbitrary Riemannian manifolds in \cite{faghfouri,olteanu}.

In this paper, we consider doubly twisted products which are a generalization of singly warped products and doubly warped products. After a
preliminary section containing basic notions of Riemannian submanifold theory, warped product and doubly twisted products, we extend in Section \ref{main} an inequality of Chen involving the squared mean curvature vector obtained in \cite{chen5,faghfouri} for doubly twisted product submanifolds. We finally end the paper by a question: May Theorems by Moore and Nolker (see \cite{moore} and \cite{nolker} for more details) be extended to doubly twisted product immersions?

\section{Preliminaries} \label{Prel}

\subsection{Background on submanifolds}

Let $N$ and $M$ be two differentiable manifolds of dimensions $n$ and $m$, respectively. We say that a differentiable map $\phi :N\rightarrow M$ is an
\textit{immersion} if the differential $\phi _{\ast }(p):T_{p}N\rightarrow T_{\phi (p)}M$ is injective for all $p\in N$. An immersion $\phi
:N\rightarrow M$ between two Riemannian manifolds with metrics $g_{N}$ and $g_{M}$, respectively, is called an \textit{isometric immersion} if
$$
g_{M}(\phi _{\ast }X,\phi _{\ast }Y)=g_{N}(X,Y),
$$ 
for every $p\in N$ and for all vector fields $X,Y$ tangent to $N$. One of the most fundamental problems in submanifold theory is the immersibility of
a Riemannian manifold in a Euclidean space (or, more generally, in a space form). According to the well-known theorem of Nash, every Riemannian $n$-manifold admits an isometric immersion into the Euclidean space $\mathbb{E}^{n(n+1)(3n+11)/2}$. In general, there exist enormously many isometric
immersions from a Riemannian manifold into Euclidean spaces if no restriction on the codimension is made \cite{chenbook}. For a submanifold of a
Riemannian manifold, there are associated several extrinsic invariants beside its intrinsic invariants. Among the extrinsic invariants, the mean curvature
vector and shape operator are the most fundamental ones \cite{chenbook}.

Let $\phi :N\rightarrow M$ be an isometric immersion of a Riemannian manifold $(N,g_{N})$ into a Riemannian manifold $(M,g_{M})$. The formulas of
Gauss and Weingarten are given respectively by
\begin{align}
\overline{\nabla }_{X}Y  =\nabla _{X}Y + h(X,Y),\;\;
\overline{\nabla }_{X}\eta =-A_{\eta }X+D_{X}\eta,
\end{align} 
for all vectors fields $X,Y$ tangent to $N$ and $\eta $ normal to $N$, where $\nabla $ and $\overline{\nabla }$ denote the Levi-Civita connections on $N$
and $M$, respectively, $h$ is the second fundamental form, $D$ the normal connection and $A$ the shape operator of $\phi $. The shape operator and the
second fundamental form are related by
\begin{equation*}
g_{M}(A_{\eta }X,Y)=g_{M}(h(X,Y),\eta ).
\end{equation*} 
Moreover, the mean curvature vector $H$ of the submanifold $N$ is defined by
\begin{equation*}
H=\frac{1}{n}\mathrm{trace}\,h=\frac{1}{n}\sum_{i=1}^{n}h(e_{i},e_{i}),
\end{equation*} 
where $(e_{1},\cdots ,e_{n})$ is a local orthonormal frame of the tangent bundle $TN$ of $N$. The squared mean curvature is given by
$
\left\Vert H\right\Vert ^{2}=\left\langle H,H\right\rangle.
$

A submanifold $N$ is said to be \textit{minimal} in $M$ if the mean curvature vector of $N$ in $M$ vanishes identically. A submanifold $N$ in a Riemannian manifold $M$ is called \textit{totally geodesic} if its second fundamental form $h$ vanishes identically. It is said to be \textit{totally umbilical} if its second fundamental form $h$ satisfies 
$
h  = g_{M}\otimes H.
$

\subsection{Warped product manifolds}

Let $(N_{1},g_{N_{1}})$ and $(N_{2},g_{N_{2}})$ be two Riemannian manifolds of dimensions $n_{1}$ and $n_{2}$, respectively, and let $\sigma $ be a
positive differentiable function on $N_{1}$. The \textit{warped product} $N_{1}\times _{\sigma }N_{2}$ is defined to be the product manifold $%
N_{1}\times N_{2}$ equipped with the Riemannian metric given by
\begin{equation*}
g_{N_{1}}+\sigma ^{2}g_{N_{2}}.
\end{equation*} 
This notion was introduced by Bishop and O'Neill for constructing negatively curvature manifolds \cite{bishopetoneil}. B. O'Neill in \cite{oneil}
discussed warped products and explored curvature formulas of warped products in terms of curvatures of components of warped product. The warped products
play some important role in differential geometry as well as in physics. Later this notion has been extended to a doubly warped product, a twisted
product, a doubly twisted product and a multiply warped product manifolds.

Let $M_{1}\times _{\rho }M_{2}$ be a Riemannian warped product manifold and $\phi _{i}:N_{i}\rightarrow M_{i}$, $i=1,2$ be isometric immersions between Riemannian manifolds. Define a positive function $\sigma =\rho \circ \phi_{1}$. Then the map
$ 
\phi :N_{1}\times _{\sigma }N_{2}\rightarrow M_{1}\times _{\rho }M_{2},
$ 
given by
$ 
\phi (x_{1},x_{2})=(\phi _{1}(x_{1}),\phi _{2}(x_{2})),
$ 
is an isometric immersion, which is called a \textit{warped product immersion} \cite{chen5}. This notion appeared in several recent studies related to
different geometric aspects. For examples, it appeared in the study of multi-rotation surfaces in \cite{DillenNolker}, a decomposition problem in
\cite{nolker}, a geometric inequality and minimal immersion problem in \cite{chen2}, and also in the study done by M. Dajczer \textit{et al.} \cite{dajczer}. Chen also studied in \cite{chen2,chen5} the fundamental geometry properties of warped
product immersions.

\subsection{Twisted product manifolds}

Let $(N_{1},g_{N_{1}})$ and $(N_{2},g_{N_{2}})$ be two Riemannian manifolds of dimensions $n_{1}$ and $n_{2}$, respectively, and let $\pi_{1}:N_{1}\times N_{2}\rightarrow N_{1}$ and $\pi _{2}:N_{1}\times N_{2}\rightarrow N_{2}$ be the canonical projections. Also, let $\sigma
_{1}:N_{1}\times N_{2}\rightarrow \left( 0,\infty \right) $ and $\sigma_{2}:N_{1}\times N_{2}\rightarrow \left( 0,\infty \right) $ be positive
differentiable functions. The \textit{doubly twisted product} of Riemannian manifolds $(N_{1},g_{N_{1}})$ and $(N_{2},g_{N_{2}})$ with twisting
functions $\sigma _{1}$ and $\sigma _{2}$ is the product manifold $N=N_{1}\times N_{2}$ equipped with the metric tensor $g_{N}=\sigma
_{2}{}^{2}g_{N_{1}}\oplus \sigma _{1}{}^{2}g_{N_{2}}$ given by
\begin{equation*}
g_{N}=\sigma _{2}{}^{2}\pi _{1}^{\ast }g_{N_{1}}+\sigma _{1}{}^{2}\pi_{2}^{\ast }g_{N_{2}}.
\end{equation*}%
We denote the Riemannian manifold $(N,g_{N})$ by $N_{1}\times _{(\sigma_{1},\sigma _{2})}N_{2}$. In particular, if $\sigma _{2}=1$ is constant,
then $N_{1}\times _{(\sigma _{1},\sigma _{2})}N_{2}$ is called the \textit{twisted product} of $(N_{1},g_{N_{1}})$ and $(N_{2},g_{N_{2}})$ with
twisting function $\sigma _{1}$. Moreover, if $\sigma _{1}$ depends only on $N_{1}$, then $N_{1}\times _{\sigma _{1}}N_{2}$ is called \textit{warped
product} of $(N_{1},g_{N_{1}})$ and $(N,g_{N_{2}})$ with the warping function $\sigma _{1}$.

Ponge and Reckziegel in \cite{ponge} mentioned that the conformal change of a Riemannian metric can be interpreted as a twisted product, namely, one where the first factor $M_1$
consists of one point only. Therefore, formulas and assertions for doubly twisted products are applicable in many situations.

For a vector field $X$ on $N_{1}$, the lift of $X$ to $N_{1}\times _{(\sigma_{1},\sigma _{2})}N_{2}$ is the vector field $\tilde{X}$ whose value at each
$(p,q)$ is the lift of $X_{p}$ to $(p,q)$. Thus the lift of $X$ is the unique vector field on $N_{1}\times _{(\sigma _{1},\sigma _{2})}N_{2}$ that
is $\pi _{1}$-related to $X$ and $\pi _{2}$-related to the zero vector field on $N_{2}$. For a doubly twisted product $N_{1}\times _{(\sigma _{1},\sigma
_{2})}N_{2}$, let $\mathcal{D}_{i}$ denote the distribution obtained from the vectors tangent to the horizontal lifts of $N_{i}$. Denote by $\nabla $ and $\nabla ^{0}$ the Levi-Civita connections on $N_{1}\times N_{2}$ associated with the doubly twisted product metric $g_{N}=\sigma _{2}{}^{2}g_{N_{1}}\oplus \sigma _{1}{}^{2}g_{N_{2}}$ and with the direct product metric $g_{0}=g_{N_{1}}+g_{N_{2}}$, respectively. We have the following.
\begin{proposition}
\label{Propos1} The Levi-Civita connections $\nabla $ and $\nabla ^{0}$ are related by
\begin{align}\label{Conn1}
\nabla_X Y &= \nabla^{0}_{X} Y + X(\ln \sigma_2)Y + Y(\ln \sigma_2)X\nonumber\\
& - g_{N_1}(X,Y) \mathrm{grad} (\ln \sigma_2), \\\label{Conn2}
\nabla_V W &=  \nabla^{0}_{V} W + V(\ln \sigma_1)W + W(\ln \sigma_1)V \nonumber\\
&- g_{N_2}(V,W) \mathrm{grad}(\ln \sigma_1),\\\label{Conn3}
\nabla_X V &=  \nabla_V X= V(\ln \sigma_1)X +  X(\ln \sigma_2)V, 
 \end{align}
for any $X$, $Y \in \mathcal{D}_1$ and $V$, $W \in \mathcal{D}_2$.
\end{proposition}
\begin{proof}
We set
$
\nabla _{X}Y=\nabla _{X}^{0}Y+\theta_{X}Y,
$
for all $X$, $Y\in \Gamma (T(N_{1}\times N_{2}))$. Then, it is easy to verify that $\theta$ is symmetric, that is, $\theta _{X}Y=\theta_{Y}X$. For any $X$, $Y$, $Z\in \Gamma (T(N_{1}\times N_{2}))$, we have $(\nabla_{X}g)(Y,Z)=0$, that is,
\begin{align*}
0& =X(g(Y,Z))-g(\nabla _{X}Y,Z)-g(Y,\nabla _{X}Z) \\
& =X(g(Y,Z))-g(\nabla _{X}^{0}Y,Z)-g(\theta _{X}Y,Z)-g(Y,\nabla
_{X}^{0}Z)-g(Y,\theta _{X}Z).
\end{align*} 
Since $g_{N}=\sigma _{2}{}^{2}g_{N_{1}}$ along $N_{1}$, so for all $X,Y,Z\in \Gamma(TN_{1})$, we have
\begin{equation*}
0=X(\sigma _{2}{}^{2})g_{N_{1}}(Y,Z)-\sigma _{2}{}^{2}\left\{g_{N_{1}}(\theta _{X}Y,Z)+g_{N_{1}}(Y,\theta _{X}Z)\right\}.
\end{equation*} 
That is, one obtains
\begin{equation}
g_{N_{1}}(\theta _{X}Y,Z)+g_{N_{1}}(Y,\theta _{X}Z)=2X(\ln \sigma
_{2})g_{N_{1}}(Y,Z).  \label{Permu1}
\end{equation} 
A circular permutation in (\ref{Permu1}) gives
\begin{equation}
g_{N_{1}}(\theta _{Y}Z,X)+g_{N_{1}}(Z,\theta _{Y}X)=2Y(\ln \sigma_{2})g_{N_{1}}(Z,X),  \label{Permu2}
\end{equation}
\begin{equation}
g_{N_{1}}(\theta _{Z}X,Y)+g_{N_{1}}(X,\theta _{Z}Y)=2Z(\ln \sigma_{2})g_{N_{1}}(X,Y).  \label{Permu3}
\end{equation} 
Putting the pieces above using the operation (\ref{Permu1}) + (\ref{Permu2}) - (\ref{Permu3}), we have
\begin{align*}
g_{N_{1}}(\theta _{X}Y,Z)& =X(\ln \sigma _{2})g(Y,Z)+Y(\ln \sigma _{2})g(X,Z)
\\
& \;\;\;-g_{N_{1}}(X,Y)g_{N_{1}}(\mathrm{grad}\ln \sigma _{2},Z).
\end{align*}%
Thus
$ 
\theta _{X}Y=X(\ln \sigma _{2})Y + Y(\ln \sigma _{2})X-g_{N_{1}}(X,Y)\mathrm{grad}\ln \sigma_{2},
$  
which proves the relation (\ref{Conn1}). Similarly, we can obtain the relation (\ref{Conn2}). The relation (\ref{Conn3}) follows directly from the formula given in \cite[Proposition 2]{ponge} using the fact that $[X, V]=0$, for any $X \in \mathcal{D}_1$ and $V\in \mathcal{D}_2$.
\end{proof}

In \cite{ponge}, Ponge and Reckziegel gave a characterization of a twisted product pseudo-Riemannian manifold in terms of distributions defined on the
manifolds. In \cite{lopez}, Fernadez-Lopez \textit{et al.} gave a condition for a twisted product manifold to be a warped product manifold by
using the Ricci tensor of the manifold. Similar characterizations were given by Kazan and Sahin in \cite{kazan} by imposing certain conditions on the
Weyl conformal curvature tensor and the Weyl projective tensor of the manifold.

\section{The main results}\label{main}

Let $M_{1}\times _{(\rho _{1},\rho _{2})}M_{2}$ be a doubly twisted product of two Riemannian manifolds $M_{1}$ and $M_{2}$ equipped with Riemannian metrics $g_{M_{1}}$ and $g_{M_{2}}$, respectively, where $\rho _{1}$ and $\rho _{2}$ are two positive smooth functions defined on $M_{1}\times M_{2}$. Let $(\phi _{1},\phi_{2}): N_{1}\times N_{2}\rightarrow M_{1}\times M_{2}$ be a direct product immersion. Define positive functions $\sigma _{i},i=1,2$ on $N_{1}\times N_{2}$ by $\sigma_{i}=\rho_{i}\circ (\phi_{1}\times \phi_{2})$. The map
$
\phi :N_{1}\times_{(\sigma _{1},\sigma _{2})}N_{2}\rightarrow M_{1}\times_{(\rho _{1},\rho _{2})}M_{2},
$
given by $\phi (x_{1},x_{2})=(\phi_{1}(x_{1}),\phi_{2}(x_{2}))$ is an isometric immersion, which is called a \textit{doubly twisted product isometric immersion}.

Denote by $\overline{\mathcal{D}}_{1}$ and $\overline{\mathcal{D}}_{2}$ the distributions obtained from vectors tangent to the horizontal lifts of $%
M_{1} $ and $M_{2}$, respectively. Denote by $\overline{\nabla }$ and $\overline{\nabla }^{0}$ the Levi-Civita connections of $M_{1}\times M_{2}$
associated with the doubly twisted product metric $g_{M}=\rho_{2}{}^{2}g_{M_{1}}+\rho _{1}{}^{2}g_{M_{2}}$ and with the direct product
metric $\overline{g}_{0}=g_{M_{1}}+g_{M_{2}}$, respectively. Then, from Proposition \ref{Propos1}, the Levi-Civita connections $\overline{\nabla }$
and $\overline{\nabla }^{0}$ are related as
\begin{align}\label{levicivitaformulas1}
 \overline{\nabla}_X Y &=  \overline{\nabla}^{0}_{X} Y + X(\ln \rho_2)Y + Y(\ln \rho_2)X\nonumber\\
 & - g_{M}(X,Y) \mathrm{grad}(\ln \rho_2),\\\label{levicivitaformulas2}
  \overline{\nabla}_V W &=  \overline{\nabla}^{0}_{V} W + V(\ln \rho_1)W + W(\ln \rho_1)V \nonumber\\
 & - g_{M}(V,W) \mathrm{grad} (\ln \rho_1),\\\label{levicivitaformulas3}
  \overline{\nabla}_X V &=  \overline{\nabla}_{V} X= V(\ln \rho_1)X + X(\ln \rho_2)V, 
 \end{align}
for all $X,Y\in \overline{\mathcal{D}}_{1}$ and $V,W\in \overline{\mathcal{D}}_{2}$, (cf. \cite{nolker,ponge}).

Let $h^{\phi }$ denote the second fundamental form of a doubly twisted product immersion $\phi :N_{1}\times_{(\sigma _{1},\sigma_{2})}N_{2}\rightarrow M_{1}\times_{(\rho _{1},\rho _{2})}M_{2}$, and let $h^{0}$ denote the second fundamental form of the corresponding direct product immersion $(\phi_{1},\phi _{2}):N_{1}\times N_{2}\rightarrow M_{1}\times M_{2}$, respectively. By applying (\ref{levicivitaformulas1}), (\ref{levicivitaformulas2}), (\ref{levicivitaformulas3}) and Gauss formula, we obtain
\begin{align}\label{Hphi1}
 h^{\phi}(X,Y) &=  h^0(X,Y) + X(\ln \rho_2)Y +Y(\ln \rho_2)X \nonumber\\
 & -g_M(X,Y)D\ln \rho_2,\\\label{Hphi2}
 h^{\phi}(V,W) & =  h^0(V,W) + V(\ln \rho_1)W +W(\ln \rho_1)V \nonumber\\
 & -g_M(V,W)D\ln \rho_1,\\
 h^{\phi}(X,V)&= 0,
 \end{align}
for all $X,Y\in \mathcal{D}_{1}$ and $V,W\in \mathcal{D}_{2}$.

The restriction of $h^{0}$ to $\mathcal{D}_{1}$ and to $\mathcal{D}_{2}$ are the second fundamental forms of $\phi _{1}:N_{1}\rightarrow M_{1}$ and $\phi
_{2}:N_{2}\rightarrow M_{2}$, respectively. Hence, $h^{0}(X,Y)$ and $h^{0}(V,W)$ are mutually orthogonal for $X,Y\in \mathcal{D}_{1}$ and $V,W\in
\mathcal{D}_{2}$.

Let $\phi :N_{1}\times _{(\sigma _{1},\sigma _{2})}N_{2}\rightarrow M$ be an isometric immersion of a doubly twisted product $N_{1}\times _{(\sigma
_{1},\sigma _{2})}N_{2}$ into a Riemannian manifold $M$ of constant sectional curvature $c$. Let $h_{1}$ and $h_{2}$ denote the restrictions of
the second fundamental form $h^{\phi }$ to $\mathcal{D}_{1}$ and $\mathcal{D}_{2}$ respectively. Then $\phi $ is called $N_{i}$\textit{-totally geodesic}
if the partial second fundamental forms $h_{i}$, $i=1,2$ vanish identically. Moreover, $\phi$ is called \textit{mixed totally geodesic} if its second fundamental form $h$ satisfies $h(X,V)=0$, for any $X\in \mathcal{D}_1$ and $V\in \mathcal{D}_2$ (see \cite{chenbook} for more details and reference therein).  
\begin{theorem}
Let $\phi :N_{1}\times _{(\sigma _{1},\sigma _{2})}N_{2}\rightarrow
M_{1}\times _{(\rho _{1},\rho _{2})}M_{2}$ be a doubly twisted product
immersion with $\mathrm{dim}N_{1}=n_{1}$ and $\mathrm{dim}N_{2}=n_{2}$ and
\begin{align*}
&\Psi(\rho_1, \rho_2)  =  2\sum_{i,j=1}^{n_1} 2g_N\left\{h^0(e_i,e_j),e_i(\ln \rho_2)e_j + e_j(\ln \rho_2)e_i\right\}\\
&  +\sum_{i,j=1}^{n_1} g_N\left\{e_i(\ln \rho_2)e_j+e_j(\ln \rho_2)e_i,e_i(\ln \rho_2)e_j+e_j(\ln \rho_2)e_i\right\}\\
& -2 \sum_{i,j=1}^{n_1}g_N\left\{h^0(e_i,e_j)+e_i(\ln \rho_2)e_j+e_j(\ln \rho_2)e_i,g_N(e_i,e_j)D\ln \rho_2\right\}\\
&  + 2\sum_{\alpha,\beta=n_1+1}^{n_1+n_2}g_N\left\{h^0(e_{\alpha},e_{\beta}),e_{\alpha}(\ln \rho_1)e_{\beta} + e_{\beta}(\ln \rho_1)e_{\alpha}\right\}\\
& + \sum_{\alpha,\beta=n_1+1}^{n_1+n_2} g_N\left\{e_{\alpha}(\ln \rho_1)e_{\beta} +e_{\beta}(\ln \rho_1)e_{\alpha},e_{\alpha}(\ln \rho_1)e_{\beta} + e_{\beta}(\ln \rho_1)e_{\alpha}\right\}\\
&  -2\sum_{\alpha,\beta=n_1+1}^{n_1+n_2}g_N\left\{h^0(e_{\alpha},e_{\beta}) + e_{\alpha}(\ln \rho_1)e_{\beta} + e_{\beta}(\ln\rho_1)e_{\alpha},g_N(e_{\alpha},e_{\beta})D\ln \rho_1\right\}. 
\end{align*} 
Then the following statements hold:
\begin{enumerate}
\item[(i)] The squared norm of the second fundamental form of $\phi$ satisfies
\begin{equation*}
\Vert h^{\phi }\Vert ^{2}\geq n_{1}\Vert D(\ln \rho _{2})\Vert^{2}+n_{2}\Vert D(\ln \rho _{1})\Vert ^{2}+\Psi (\rho _{1},\rho_{2}).
\end{equation*}
\item[(ii)] $\phi $ is $N_{1}$-totally geodesic if and only $\phi_{1}:N_{1}\rightarrow M_{1}$ is totally geodesic and $X(\ln \rho_{2})Y+Y(\ln \rho _{2})X=g_{M}(X,Y)D\ln \rho _{2}$.
\item[(iii)] $\phi $ is $N_{2}$-totally geodesic if and only $\phi_{2}:N_{2}\rightarrow M_{2}$ is totally geodesic and $V(\ln \rho
_{1})W+W(\ln \rho _{1})V=g_{M}(V,W)D\ln \rho _{1}$.
\item[(iv)] $\phi $ is a totally geodesic immersion if and only if $\phi $ is both $N_{i}$-totally geodesics.
\end{enumerate}
\end{theorem}
\begin{proof}
(i) is proven as follows. Let $e_{i}\in \mathcal{D}_{i}$, $i=1,\cdots ,n_{1}$ and $e_{\alpha }\in\mathcal{D}_{\alpha }$, $\alpha =n_{1}+1,\cdots,n_{1}+n_{2}$ be orthonormal frame fields of $N_{1}\times _{(\sigma _{1},\sigma _{2})}N_{2}$. Then we have
\begin{align*}
\Vert h^{\phi }\Vert ^{2} &= \sum_{a,b=1}^{n_{1}+n_{2}}g_{N}\left( h^{\phi}(e_{a},e_{b}),h^{\phi }(e_{a},e_{b})\right) \\ 
&= \sum_{i,j=1}^{n_{1}}g_{N}\left(h^{0}(e_{i},e_{j}),h^{0}(e_{i},e_{j})\right) \\
& +\sum_{\alpha ,\beta =n_{1}+1}^{n_{1}+n_{2}}g_{N}\left( h^{0}(e_{\alpha},e_{\beta }),h^{0}(e_{\alpha },e_{\beta })\right) \\
& +\sum_{i,j=1}^{n_{1}}g_{N}\left( g(e_{i},e_{j})D\ln \rho_{2},g(e_{i},e_{j})D\ln \rho _{2}\right) \\
& +\sum_{\alpha ,\beta =n_{1}+1}^{n_{1}+n_{2}}g_{N}\left(g(e_{\alpha}, e_{\beta })D\ln \rho _{1},g(e_{\alpha },e_{\beta })D\ln \rho _{1}\right) \\
& +\Psi (\rho _{1},\rho _{2}).
\end{align*}
 If $\phi:N_1 \times_{(\sigma_1,\sigma_2)} N_2 \rightarrow M$ is a $N_i$-totally geodesic immersion, then it follows from (\ref{Hphi1}) and (\ref{Hphi2}) that
 \begin{align*}
  0 &=  h^0(X,Y) + X(\ln \rho_2)Y +Y(\ln \rho_2)X -g_M(X,Y)D\ln \rho_2,\\
  0 &=  h^0(V,W) + V(\ln \rho_1)W +W(\ln \rho_1)V -g_M(V,W)D\ln \rho_1.
 \end{align*}
Since $h^0(X,Y), D\ln \rho_2$ and $h^0(V,W), D\ln \rho_1$ are orthogonal, respectively, we have
\begin{eqnarray*}
 h^0(X,Y) =0, \quad h^0(V,W) =0,
\end{eqnarray*}
that is, $\phi_1$ and $\phi_2$ are totally geodesic, and
\begin{align*}
 0 &=  X(\ln \rho_2)Y +Y(\ln \rho_2)X -g_M(X,Y)D\ln \rho_2,\\
  0 &=  V(\ln \rho_1)W +W(\ln \rho_1)V -g_M(V,W)D\ln \rho_1.
\end{align*}
Conversely, if $\phi_1$ and $\phi_2$ are totally geodesic and 
\begin{align*}
 0 & = X(\ln \rho_2)Y +Y(\ln \rho_2)X -g_M(X,Y)D\ln \rho_2,\\
  0 & = V(\ln \rho_1)W +W(\ln \rho_1)V -g_M(V,W)D\ln \rho_1.
\end{align*}
It follows from (\ref{Hphi1}) and (\ref{Hphi2}) that
$$
h^{\phi}(X,Y)= 0 \;\;\mbox{and}\;\; h^{\phi}(V,W)=0,
$$
for $X,Y\in \mathcal{D}_1$ and $V,W\in \mathcal{D}_2$. This completes the proof of (ii) and (iii). The statement (iv) follows from (ii) and (iii). 
\end{proof}

For a doubly warped product immersion, that is, when $\sigma_1$ (respectively $\sigma_2$) depends only on $N_1$ (respectively $N_2$) and $\rho_1$ (respectively 
$\rho_2$) depends only on $M_1$ (respectively $M_2$), then we have the following
result by Faghfouri and Majidi \cite{faghfouri}.
\begin{corollary}\cite[Theorem 1]{faghfouri} 
Let $\phi :(\phi _{1},\phi _{2}):N_{1}\times_{(\sigma_{1},\sigma_{2})}N_{2}\rightarrow M_{1}\times _{(\rho _{1},\rho _{2})}M_{2}$ be a
doubly warped product immersion between two doubly warped product manifolds. Then the following statements are true:
\begin{enumerate}
\item[(i)] $\phi$ is mixed totally geodesic.
\item[(ii)] The squared norm of the second fundamental form of $\phi $ satisfies
\begin{equation*}
\left\Vert h^{\phi }\right\Vert ^{2}\geq n_{1}\left\Vert D\ln \rho
_{2}\right\Vert ^{2}+n_{2}\left\Vert D\ln \rho _{1}\right\Vert ^{2}
\end{equation*}
with equality holding if and only if $\phi _{1}:N_{1}\rightarrow M_{1}$ and $\phi _{2}:N_{2}\rightarrow M_{2}$ are both totally geodesic immersions.
\item[(iii)] $\phi $ is $N_{1}$-totally geodesic if and only if $\phi_{1}:N_{1}\rightarrow M_{1}$ is totally geodesic and $D\ln \rho _{2}=0$.
\item[(iv)] $\phi $ is $N_{2}$-totally geodesic if and only if $\phi_{2}:N_{2}\rightarrow M_{2}$ is totally geodesic and $D\ln \rho _{1}=0$.
\item[(v)] $\phi $ is a totally geodesic immersion if and only if $\phi$ is both $N_{1}$-totally geodesic and $N_{2}$-totally geodesic.
\end{enumerate}
\end{corollary}
In the case of warped product immersion, we have the following result of Chen.
\begin{corollary}\label{Chen} \cite[Theorem 1]{chen5} 
Let $\phi :(\phi _{1},\phi_{2}):N_{1}\times_{\sigma}N_{2}\rightarrow M_{1}\times _{\rho }M_{2}$ be a
warped product immersion between two warped product manifolds. Then the following statements are true:
\begin{enumerate}
\item[(i)] $\phi$ is mixed totally geodesic.
\item[(ii)] The squared norm of the second fundamental form of $\phi $ satisfies
\begin{equation*}
\left\Vert h^{\phi }\right\Vert ^{2}\geq n_{2}\left\Vert D\ln \rho
\right\Vert ^{2}
\end{equation*}
with equality holding if and only if $\phi_{1}:N_{1}\rightarrow M_{1}$ and $\phi _{2}:N_{2}\rightarrow M_{2}$ are both totally geodesic immersions.
\item[(iii)] $\phi $ is $N_{1}$-totally geodesic if and only if $\phi_{1}:N_{1}\rightarrow M_{1}$ is totally geodesic.
\item[(iv)] $\phi $ is $N_{2}$-totally geodesic if and only if $\phi_{2}:N_{2}\rightarrow M_{2}$ is totally geodesic and $(\nabla \ln \rho)|_{N_{1}}=\nabla \ln f$ holds, that is, the restriction of the gradient of $\ln \rho $ to $N_{1}$ is the gradient of $\ln f$, or equivalently, $D\ln\rho =0$.
\item[(v)] $\phi $ is a totally geodesic immersion if and only if $\phi $ is both $N_{1}$-totally geodesic and $N_{2}$-totally geodesic.
\end{enumerate}
\end{corollary}
Let $\phi :N_{1}\times_{(\sigma _{1},\sigma _{2})}N_{2}\rightarrow M$ be an isometric immersion of a doubly twisted product $N_{1}\times _{(\sigma
_{1},\sigma _{2})}N_{2}$ into a Riemannian manifold $M$ with constant sectional curvature $c$. Denote by $\mathrm{trace}\,h_{1}$ and $\mathrm{trace%
}\,h_{2}$ the trace of $h_{i}$, $i=1,2$ restricted to $N_{1}$ and $N_{2}$, respectively, that is
\begin{equation*}
\mathrm{trace}\,h_{1}=\sum_{\alpha =1}^{n_{1}}h(e_{\alpha },e_{\alpha}),\;\; \mbox{and}\;\; \mathrm{trace}\,h_{2}=\sum_{t=n_{1}+1}^{n_{1}+n_{2}}h(e_{t},e_{t}),
\end{equation*}
for some orthonormal frame fields $e_{1},\cdots ,e_{n_{1}}$ and $e_{n_{1}+1},\cdots ,e_{n_{1}+n_{2}}$ of $\mathcal{D}_{1}$ and $\mathcal{D}_{2}$, respectively. The partial mean curvature vectors $H_{i}$ is defined
by
\begin{equation*}
H_{1}=\frac{1}{n_{1}}\mathrm{trace}\,h_{1}\quad \mathrm{and}\quad H_{2}=\frac{1}{n_{2}}\mathrm{trace}\,h_{2}.
\end{equation*} 
An immersion $\phi :N_{1}\times _{(\sigma _{1},\sigma _{2})}N_{2}\rightarrow M$ is called $N_{i}$\textit{-minimal} if the partial mean curvatures $H_{i},i=1,2$ vanish identically \cite{chenbook}.

In the sequel, we need the following lemma.
\begin{lemma}\cite{gos}\label{LemmaD}
 Let $f$ be a smooth function that take its arguments from some product space $N_{1}\times N_{2}$. Then, normal component $Df$ of the gradient of $f$ on $N_{1}\times N_{2}$ can be decomposed as 
 $$
 Df = D_{1}f + D_{2}f,
 $$
 where $D_{i}f$ are normal components of the gradients of $f$ on $N_{i}$, $i=1,2$.
\end{lemma} 
\begin{theorem}
Let $\phi :N_{1}\times _{(\sigma _{1},\sigma _{2})}N_{2}\rightarrow M_{1}\times _{(\rho _{1},\rho _{2})}M_{2}$ be a doubly twisted product
immersion between two doubly twisted product manifolds. Then the following statements are true:
\begin{enumerate}
\item[(i)] $\phi$ is $N_1$-minimal if and only $\phi_1:N_1\rightarrow M_1$ is a minimal isometric, $\displaystyle n_{1} \sigma_{2}^{2} D_{1}\ln \rho _{2} = 2\sum_{i=1}^{n_{1}}e_{i}(\ln\rho _{2})e_{i}\;\;$ and $\;\;D_{2}\ln \rho _{2}=0$.
\item[(ii)]  $\phi$ is $N_2$-minimal if and only $\phi_2:N_2\rightarrow M_2$ is a
minimal isometric, $\displaystyle n_2 \sigma_{1}^{2} D_2\ln \rho _{1} = 2\sum_{\alpha=1}^{n_{2}}e_{\alpha}(\ln\rho _{1})e_{\alpha}\;\;$
and $\;\;D_{1}\ln \rho_{1}=0$.
\item[(iii)]  $\phi $ is a minimal immersion if and only if the mean curvature
vectors of $\phi _{1}$ and $\phi _{2}$ are given by 
\begin{equation*}
n_{1}^{-1} n_2\sigma^{2}_{1}D_1\ln \rho_{1} + \sigma^{2}_{2}D_1\ln \rho_{2} -2n_{1}^{-1}\sum_{a=1}^{n_{1}}e_{a}(\ln \rho_{2})e_{a},
\end{equation*}
and 
\begin{equation*}
n_{1}n_{2}^{-1} \sigma^{2}_{1}D_2\ln \rho_{2} +  \sigma^{2}_{1}D_2\ln \rho_{1}-2n_{2}^{-1}\sum_{\alpha=1}^{n_{2}}e_{\alpha}(\ln \rho_{2})e_{\alpha},
\end{equation*} 
respectively.
\end{enumerate}
\end{theorem}
\begin{proof}
Let $e_{i}\in \mathcal{D}_{1},i=1,\cdots ,n_{1}$ and $e_{\alpha }\in \mathcal{D}_{2},\alpha =n_{1}+1,\cdots ,n_{1}+n_{2}$ 
be orthonormal frame fields of $N_{1}\times _{(\sigma _{1},\sigma _{2})}N_{2}$. Recall that the partial mean curvature $H_i$
is defined by 
\begin{equation*}
H_{i}=\frac{1}{n_{i}}\mathrm{trace}\,h_{i}.
\end{equation*} 
(i) If $\phi:N_1 \times_{(\sigma_1,\sigma_2)} N_2 \rightarrow M_1 \times_{(\rho_1,\rho_2)} M_2$ is $N_1$-minimal, then it follows from equation (4.4) that
 \begin{equation*}
H_{1}=\frac{1}{n_{1}}\sum_{i=1}^{n_{1}}\left( h^{0}(e_{i},e_{i})+2e_{i}(\ln\rho _{2})e_{i}-g_{M}(e_{i},e_{i})D\ln \rho _{2}\right)=0. 
\end{equation*}
From Lemma \ref{LemmaD}, we have $D\ln \rho _{2}=D_1\ln \rho_{2} + D_2\ln \rho_{2}$ and 
\begin{equation*}
\frac{1}{n_{1}}\sum_{i=1}^{n_{1}}h^{0}(e_{i},e_{i}) + 2\frac{1}{n_{1}}\sum_{i=1}^{n_{1}}e_{i}(\ln\rho _{2})e_{i}
- D_1\ln \rho _{2}- D_{2}\ln \rho _{2}=0.
\end{equation*}
That is,
\begin{equation*}
 0= \frac{\mathrm{trace}\,h^{0}_{1}}{n_1} + 2\frac{1}{n_{1}}\sum_{i=1}^{n_{1}}e_{i}(\ln\rho _{2})e_{i}
- D_1\ln \rho _{2}- D_{2}\ln \rho _{2}.
\end{equation*}
Since $D_1\ln \rho _{2}$ and $\mathrm{trace}\,h^{0}_{1}$ are orthogonal, we have $\phi_1$ is minimal immersion, 
$
\displaystyle n_1 D_1\ln \rho _{2} = 2\sum_{i=1}^{n_{1}}e_{i}(\ln\rho _{2})e_{i},
$
and 
$
D_{2}\ln \rho _{2}=0.
$
(ii) is similar to (i). Now we prove (iii). If $\phi:N_1 \times_{(\sigma_1,\sigma_2)} N_2 \rightarrow M_1 \times_{(\rho_1,\rho_2)} M_2$ is minimal immersion, then $\mathrm{trace}\, h=0$. By applying equations (4.4) and (4.5), we get
\begin{align*}\label{MininalRela1}
 0&=   \sum_{a=1}^{n_{1}}\left\{ h^{0}(e_{a},e_{a}) + 2e_{a}(\ln
\rho _{2})e_{a} - n_1 \sigma^{2}_{2}D_1\ln \rho _{2} - n_{1} \sigma^{2}_{2}D_2\ln \rho _{2} \right\}\nonumber\\
&  +  \sum_{\alpha=1}^{n_{2}}\left\{ h^{0}(e_{\alpha},e_{\alpha}) + 2e_{\alpha}(\ln\rho _{1})e_{\alpha} - n_2\sigma^{2}_{1}D_1\ln \rho _{1} - n_{2}\sigma^{2}_{1}D_2\ln \rho_{1} \right\},
\end{align*}
which implies that
\begin{align*} 
0&=  \mathrm{trace}h^{0}_{1} + 2\sum_{a=1}^{n_{1}}e_{a}(\ln \rho _{2})e_{a} -n_{1} \sigma^{2}_{2}D_1\ln \rho _{2} -  n_{1}  \sigma^{2}_{2}D_2\ln \rho _{2}\nonumber\\
&  +  \mathrm{trace}h^{0}_{2} + 2\sum_{a=1}^{n_{2}}e_{\alpha}(\ln \rho _{1})e_{\alpha}
-  n_{2}\sigma^{2}_{1}D_1 \ln \rho _{1} -  n_{2}\sigma^{2}_{1} D_2 \ln \rho_{1}.
\end{align*} 
Since $\mathrm{trace}\, h^{0}_{1}$, $\displaystyle\sum_{a=1}^{n_{1}}e_{a}(\ln \rho _{2})e_{a}$, $D_1\ln \rho_1$ and $D_1\ln \rho_2$ are tangent to the first factor $M_{1}$, and $\mathrm{trace}\, h^{0}_{2}$, $\displaystyle\sum_{a=1}^{n_{2}}e_{\alpha}(\ln \rho _{1})e_{\alpha}$, $D_2\ln \rho_1$ and $D_2\ln \rho_2$ are tangent to the second factor $M_{2}$, then
 \begin{equation*}
\frac{1}{n_{1}}\mathrm{trace}\, h^{0}_{1} = n_{1}^{-1}n_{2}\sigma^{2}_{1}D_1\ln \rho _{1} +  \sigma^{2}_{2}D_1\ln \rho _{2}
-\frac{2}{n_1}\sum_{a=1}^{n_{1}}e_{a}(\ln \rho _{2})e_{a}
\end{equation*}
and 
\begin{equation*}
\frac{1}{n_{2}}\mathrm{trace}\,h^{0}_{2} = \sigma^{2}_{1}D_2\ln \rho _{1} + n_{2}^{-1}n_{1} \sigma^{2}_{2}D_2\ln \rho _{2}
-\frac{2}{n_2}\sum_{\alpha=1}^{n_{2}}e_{\alpha}(\ln \rho _{2})e_{\alpha},
\end{equation*} 
which complete the proof.
\end{proof}

\section{On doubly twisted product immersions}\label{Doubly} 

A basic problem in the theory of submanifold is to provide conditions that imply that an isometric immersion of a product manifold must be a product of
isometric immersions. The first contribution to this problem was given by Moore \cite{moore}. The latter showed the following result:
\begin{theorem}
\label{Moore} \cite{moore} Let $\phi :N\rightarrow \mathbb{R}^{n}$ be an isometric immersion of a Riemannian product $N=N_{1}\times N_{2}\times
\cdots \times N_{k}$ of connected Riemannian manifolds into the Euclidean space $\mathbb{R}^{n}$. Then $\phi $ is a mixed totally geodesic immersion
if and only if $\phi $ is a product immersion, that is, there exists an isometry
$
\psi :M_{1}\times M_{2}\times \cdots \times M_{k}\rightarrow \mathbb{R}^{n},
$
and there are isometric immersions $\phi_{i}:N_{i}\rightarrow M_{i}$, $i=1,\ldots ,k$, such that
$
\phi =\psi \circ (\phi _{1}\times \phi _{2}\times \cdots \times \phi _{k}).
$
\end{theorem}

In \cite{ferus}, Ferus used Moore's Theorem to factorize isometric immersions into $\mathbb{R}^{n}$ with parallel second fundamental form. This
was an important step towards their classification. In \cite{DillenNolker}, Dillen and Nolker classified normally flat semi-parallel (a condition
slightly weaker than that of a parallel second fundamental form) submanifolds of a warped product manifolds. Nolker extended Moore's result
to isometric immersions of multiply warped products.

\begin{theorem}
\label{Nolker} \cite{nolker} Let $\phi :N_{1}\times _{\sigma_{2}}N_{2}\times \cdots \times _{\sigma _{k}}N_{k}\rightarrow \mathbb{R}^{m}(c)$ be an isometric immersion of a multiply warped product into a complete simply-connected real space form $\mathbb{R}^{m}(c)$ of constant
curvature $c$. Then $\phi $ is mixed totally geodesic immersion if and only if there exists an isometry
$
\psi :M_{1}\times _{\rho _{2}}M_{2}\times \cdots \times_{\rho_{k}}M_{k}\rightarrow \mathbb{R}^{m}(c),
$
where $M_{1}$ is an open subset of a standard space and $M_{2},\ldots ,M_{k}$ are standard spaces, and there exist isometric immersions $\phi_{i}:N_{i}\rightarrow M_{i}$, $i=1, \cdots, k$ such that
$
\sigma_{i}=\rho_{i}\circ \phi_{1},\;\; \mbox{for}\;\; i=2,\cdots,k\;\;\mbox{and}\;\; \phi =\psi \circ (\phi_{1}\times \phi_{2}\times\cdots\times\phi_{k}).
$
\end{theorem}
\noindent
In the paper \cite{dajczer}, Dajczer and Vlachos gave a condition under which an isometric immersion of a warped product of manifolds into a space form must be a warped product of isometric immersions.

Let $N_{1}\times _{(\sigma _{1},\sigma _{2})}N_{2}$ be a doubly twisted product of two Riemannian manifolds $N_{1}$ and $N_{2}$ equipped with
Riemannian metrics $g_{N_{1}}$ and $g_{N_{2}}$ of dimensions $n_{1}$ and $n_{2}$, respectively, where $\sigma _{1}$ and $\sigma _{2}$ are two positive
smooth functions defined on $N_{1}\times N_{2}$. Denote by $\mathcal{D}_{1}$ and $\mathcal{D}_{2}$ the distributions obtained from the vectors tangent to
$N_{1}$ and $N_{2}$, respectively, (or more precisely, from vectors tangent to the horizontal lifts of $N_{1}$ and $N_{2}$, respectively).
 
 Let $\phi:N_1 \times_{(\sigma_1,\sigma_2)} N_2 \rightarrow M$ be an isometric immersion of a doubly twisted product $N_1\times_{(\sigma_1,\sigma_2)} N_2$ into a Riemannian manifold $M$. Denote by $h$ the second fundamental form of $\phi$. 
 
We finally end this section by the following problem: \textit{Let $\phi :N_{1}\times_{(\sigma _{1},\sigma _{2})}N_{2}\rightarrow \mathbb{R}^{m}$ be an isometric immersion of a connected doubly twisted product $N=N_{1}\times _{(\sigma _{1},\sigma _{2})}N_{2}$ into a real space form $\mathbb{R}^{m}(c)$ of constant curvature $c$. Now, if $\phi $ is mixed totally geodesic, is there an isometric immersion $\psi :M_{1}\times _{(\rho_{1},\rho _{2})}M_{2}\rightarrow G$ from a doubly twisted product $M_{1}\times _{(\rho _{1},\rho _{2})}M_{2}$ onto an open, dense subset $G\subset \mathbb{R}^{m}$ and a direct product isometric immersion $(\phi_{1},\phi_{2}):N_{1}\times N_{2}\rightarrow M_{1}\times M_{2}$ such that $\sigma _{i}=\rho_{i}\circ (\phi_{1}\times \phi_{2})$ for $i=1,2$ and $\; \phi =\psi \circ (\phi_{1}\times \phi_{2})$? }

\section*{Acknowledgments}  

The authors would like to thank Professor Mukut Mani Tripathi (Banaras Hindu University, India)  for his many valuable suggestions and comments.
ASD is thankful to the University of KwaZulu-Natal for financial support. This work is based on the research supported wholly/ in part by the National Research Foundation of South Africa (Grant Numbers: 95931 and 106072). The authors thank referees for helping them to improve the presentation.

\end{document}